\documentclass[leqno,11pt,a4paper]{amsart}
\usepackage{bw}

\begin{document}
\title[Picard rank of toric Fanos with minimal curve constraints]{Bounds on the Picard rank of toric Fano varieties with minimal curve constraints}
\author{R.~Beheshti}
\address{Department of Mathematics \& Statistics \\ Washington University in St.~Louis \\ St.~Louis \\ MO \\ 63130 \\ USA}
\email{beheshti@wustl.edu}
\author{B.~Wormleighton}
\address{Department of Mathematics \& Statistics \\ Washington University in St.~Louis \\ St.~Louis \\ MO \\ 63130 \\ USA}
\email{benw@wustl.edu}
\maketitle

\begin{abstract} We study the Picard rank of smooth toric Fano varieties constrained to possess families of minimal rational curves of given degree. We discuss variants of a conjecture of Chen--Fu--Hwang and prove a version of their statement that recovers the original conjecture in sufficiently high dimension. We also prove new cases of the original conjecture for high degrees in all dimensions. Our main tools come from toric Mori theory and the combinatorics of Fano polytopes. 
\end{abstract}

\section{Introduction}

\subsection{Context}

Controlling numerical invariants of important classes of varieties is valuable for classification purposes and for understanding their geography. In this paper we study toric Fano varieties and seek to bound their Picard rank. One of the key results in this area is a theorem of Casagrande.

\begin{thm*}[\!{\cite[Thm.~1]{cas_num_06}}] \label{thm:casa} Let $X$ be a $\Q$-factorial Gorenstein toric Fano variety of dimension $n$. Then
$$\rho(X)\leq 2n.$$
\end{thm*}

As is often the case in toric geometry, the proof occurs on a combinatorial avatar of $X$, its \emph{spanning} or \emph{Fano polytope} $P$. In this setting the Picard rank of $X$ is given by the number of vertices of $P$ minus the dimension of $X$ (or of $P$).

It is natural to investigate how such an upper bound behaves in light of geometric constraints imposed on the class of Fano varieties. Recall that an irreducible component $\mathcal C$ of the space of rational curves in a Fano variety $X$ is called a \emph{minimal component} if the union of curves parametrized by $\mathcal C$ is dense in $X$ and for a general $x$ in $X$, the family $\mathcal C_x$ of curves parametrized by $\mathcal C$ through $x$ is complete. A minimal component $\mathcal C$ is said to have degree $k$ if $-K_X \cdot C=k$ for a curve $C$ parametrizd by $\mathcal C$. Work of Chen--Fu--Hwang \cite{cfh_min_14} studies the Picard rank of $X$ where $X$ is a smooth toric Fano variety of dimension $n$ with a minimal component of degree $k$.  They conjecture \cite[Thm.~7]{cfh_min_14} that in this case
\begin{equation} \tag{$\ast$} \label{eqn:conj}
\rho(X)\cdot (k-1)\leq\frac{n(n+1)}{2}.
\end{equation}
It is proved in the same paper for all $k$ in dimension at most $4$.

\subsection{Main results and conjectures} \label{sec:main}

We further study the restrictions on the Picard rank imposed by considering toric Fano varieties with minimal curve constraints. It will be convenient for us to define the \emph{codegree} $\on{codeg}(\mathcal{C})$ of a minimal component $\mathcal{C}$ to be $\on{dim}(X)+1-\on{deg}(\mathcal{C})$. For instance, if $X$ has a codegree $0$ (i.e.~degree $n+1$) minimal component then $X=\pr^n$ \cite{cms_cha_02}. We formulate the following variants of Chen--Fu--Hwang's conjecture.

\begin{conjecture*}[Strong conjecture] \label{conj:strong}
Let $X$ be a smooth toric Fano variety of dimension $n$ with a minimal component $\mathcal{C}$ of degree $k$. Then
$$\rho(X)\leq 2n-2k+4=2\on{codeg}(\mathcal{C})+2$$
\end{conjecture*}

Note that Conj.~\ref{conj:strong} is typically, but not always, stronger than (\ref{eqn:conj}). This conjecture, like that of Chen--Fu--Hwang, seems difficult and so we formulate the following weaker version.

\begin{conjecture*}[Weak conjecture] \label{conj:weak}
Let $X$ be a smooth toric Fano variety with a minimal component $\mathcal{C}$ of \textbf{codegree} $q$. Then there exists a constant $A_q$ depending only on the codegree $q$ such that
$$\rho(X)\leq A_q$$
\end{conjecture*}

It is easy to see that Conj.~\ref{conj:weak} implies (\ref{eqn:conj}) for all Fano varieties of sufficiently large dimension and all minimal components of fixed codegree (hence their degree grows linearly with dimension).

We present some results towards these conjectures. We obtain our main results by performing a detailed `Batyrev-style' analysis of \emph{primitive collections} (see Def.~\ref{def:prim_coll}) on toric Fano varieties, along with new applications of results in toric Mori theory \cite{rei_dec_83,cas_tor_03}.

\begin{thm*}[Thm.~\ref{thm:weak_main}] \label{thm:weak}
Conjecture \ref{conj:weak} is true.
\end{thm*}

We can show more for toric Fano varieties with minimal components of low codegree. It follows from classification theorems \cite{cms_cha_02,obr_cla_07} that $A_0=1$ and $A_1=3$.

\begin{thm*}[Thm.~\ref{thm:codegree2}] We verify that $\rho(X)\leq 5$ for all smooth toric Fano varieties $X$ with a minimal component of codegree $2$.
\end{thm*}

\begin{cor*}[{Thm.~\ref{thm:weak_main} + Cor.~\ref{cor:codegree2}}] We can conclude:
\begin{enumerate}
    \item Conjecture \ref{conj:strong} holds for all smooth toric Fano varieties with a minimal component of codegree at most $2$.
    \item The conjecture of Chen--Fu--Hwang holds for all smooth toric Fano varieties with a minimal component of codegree at most $2$.
    \item The conjecture of Chen--Fu--Hwang hold for smooth toric Fano varieties of sufficiently high dimension for all minimal components of fixed codegree.
\end{enumerate}
\end{cor*}

\begin{remark*}
Some other cases of Conj.~\ref{conj:strong} -- and hence the conjecture of Chen--Fu--Hwang -- can be verified using toric fibration methods; for a reference on the theory of toric fibrations see \cite{dir_lin_14}. One can utilise classification results of Casagrande \cite[Thm.~2.4]{cas_tor_03} to reduce Conj.~\ref{conj:strong} for a toric Fano variety of dimension $n$ to the same bound for a toric Fano variety of lower dimension in some situations. For instance, if $X$ contains a toric divisor $D$ with $\rho_X-\rho_D=3$ then $X$ can be expressed as a fibration with base a toric Fano variety $B$ of dimension $\on{dim}(X)-2$. If Conj.~\ref{conj:strong} holds for $B$ then one can show that it also holds for $X$. In this way one can produce new examples of (albeit somewhat special) toric Fano varieties satisfying Conj.~\ref{conj:strong}, such as all toric Fano $5$- and $6$-folds with such a toric divisor. We are curious if a melange of the techniques in this paper with other (e.g.~toric fibration) methods can be used to access more cases, or even the general case, of Conj.~\ref{conj:strong}.
\end{remark*}

\begin{remark*}
Another interesting problem is to study Conj.~\ref{conj:strong} and the conjecture of Chen--Fu-Hwang for wider classes of toric Fano varieties where some singularities are permitted. Many bounds on geometric invariants for toric Fano varieties extend beyond the smooth case, such as Thm.~\ref{thm:casa}, and so we are interested to explore versions of the conjectures addressed in this paper for other classes, including Gorenstein toric Fano varieties and toric Fano orbifolds.
\end{remark*}

\subsection*{Acknowledgements} The authors are grateful to Al Kasprzyk, Andrea Petracci, and Jun-Muk Hwang for helpful conversations. We also especially thank Cinzia Casagrande for helping us identify and correct an error in an earlier draft. RB is partially supported by NSF grant DMS-2101935.

\section{Background}

\subsection{Fans, polytopes and toric varieties}

Toric varieties are described by various combinatorial gadgets, most notably \emph{fans} and \emph{polytopes}.

A fan $\Sigma$ in a finite-dimensional real vector space $V$ is a collection of cones satisfying certain regularity and compatibility conditions \cite[\S3.1]{cls_tor_11}. Given a fan $\Sigma$ one can construct a toric variety $X_\Sigma$. The combinatorics of $\Sigma$ dictate many of the aspects of the geometry of $X_\Sigma$; for example, there is a bijection between the set of rays (one-dimensional cones in $\Sigma$) and torus-invariant prime divisors on $X_\Sigma$. We denote the set consisting of the primitive elements along the rays of $\Sigma$ by $G(\Sigma)$. Given $v_1,\dots,v_n$ we denote the cone they generate (i.e.~their $\R_{\geq0}$-linear span) by $\langle v_1,\dots, v_n\rangle$.

A polytope $P$ in a real vector space is the convex hull of a finite set of points. Consider a set $\{v_1,\dots,v_r\}$ of vertices of $P$. We say that a polytope $P$ is \emph{simplicial} if every face of $P$ is a simplex. $P$ is \emph{smooth} if the vertices of each facet of $P$ form a $\Z$-basis of $N$; in particular, $P$ is simplicial. Note that if $P$ is a simplicial polytope then a face with $r$ vertices will be $(r-1)$-dimensional.

\begin{definition} We say that a polytope $P$ is \emph{Fano} if:
\begin{enumerate}
\item every vertex of $P$ is primitive,
\item the origin is the only interior lattice point of $P$.
\end{enumerate}
\end{definition}

To such a polytope we can associate a projective toric Fano variety $X_P$ as the (possibly singular) toric variety corresponding to the complete fan $\Sigma_P$ whose rays are generated by the vertices of $P$ and such that the cone $\langle v_1,\dots,v_n\rangle$ is in $\Sigma_P$ if and only if $\langle v_1,\dots,v_n\rangle$ is a face of $P$. This justifies the slight ambiguity of notation. Here Fano means that $-K_{X_P}$ is $\Q$-Cartier and ample. Conversely, to a toric Fano variety $X$ one can associate a polytope -- the polytope corresponding to the anticanonical divisor of $X$ \cite[\S4]{cls_tor_11} -- and the polar dual of this polytope is a Fano polytope $P$ such that $X\cong X_P$. We often describe $P$ as `the Fano polytope for $X$'.

Much is known about the combinatorics of Fano polytopes and their connections with mirror symmetry, classification of Fano varieties, and other parts of combinatorics and algebraic geometry \cite{nil_gor_05,kn_fan_13,acgk_mir_16,kw_qua_18,ht_fac_17,nil_com_06,acgk_min_12,ks_com_00}.

\subsection{Primitive collections and minimal components}

\begin{definition} \label{def:prim_coll} A \emph{primitive collection} $\mP=\{\rho_1,\dots,\rho_r\}$ for a toric variety $X_\Sigma$ is a set of rays in $\Sigma$ such that the rays in $\mP$ do not together span a cone in $\Sigma$, but such that the rays in every proper subset of $\mP$ do span a cone.
\end{definition}

Primitive collections, introduced in \cite{bat_cla_91}, have been powerfully used to study the nef and effective cones of toric varieties and to approach classification problems \cite{cv_pri_09,bat_cla_99,rei_dec_83}. We often conflate rays with their primitive generators and so will also describe a set of vectors as forming a primitive collection when their corresponding rays do. When $\Sigma$ arises from a polytope $P$ we will also describe vertices of $P$ as forming primitive collections.

In our story certain primitive collections produce families of minimal rational curves. We recall the setup from \cite{cfh_min_14}. Fix a uniruled, complete, smooth variety $X$. In our context this will always be a smooth projective toric variety, usually Fano. Let $\on{RatCurves}^n(X)$ be the normalised space of rational curves on $X$ \cite{kol_rat_13}. For an irreducible component $\mathcal{C}\subseteq\on{RatCurves}^n(X)$ consider its universal family $\psi\colon \mathcal{U}\to \mathcal{C}$, which also comes with a map $\mu\colon \mathcal{U}\to X$.

\begin{definition}
$\mathcal{C}$ is a \emph{dominating component} if $\mu$ is dominant. $\mathcal{C}$ is a \emph{minimal component} if in addition $\mu^{-1}(x)$ is complete for general $x\in X$.
\end{definition}

\begin{thm}[\!{\cite[Thm.~7]{cfh_min_14}}] \label{thm:min_cpt} Let $X_\Sigma$ be a smooth projective toric variety. Minimal components in $X_\Sigma$ correspond to primitive collections $\{\rho_1,\dots,\rho_k\}$ such that the corresponding primitive generators $v_1,\dots,v_k$ satisfy
$$v_1+\dots+v_k=0$$
\end{thm}

The (anticanonical) \emph{degree} $\on{deg}(\mathcal{C})$ of a minimal component $\mathcal{C}$ is $-K_X\cdot C$ for a curve $C$ in the family $\mathcal{U}$. In the toric context this is equal to the number $k$ in Thm.~\ref{thm:min_cpt}. As stated in the introduction, we will make much use of the \emph{codegree}: $\on{codeg}(\mathcal{C}):=\on{dim}(X)+1-\on{deg}(\mathcal{C})$.

\subsection{Mori theory and extremal classes}

Following \cite{cas_con_03,rei_dec_83} we introduce the relevant aspects of toric Mori theory for toric varieties in our context. Fix a smooth toric variety $X$. The effective cones of divisors and curves on $X$ are polyhedral cones generated by the irreducible torus-invariant divisors and curves respectively (hence the corresponding nef cones are also polyhedral).

The rays of the effective cone of curves $\on{NE}(X)$ are called \emph{extremal rays} \cite[\S1]{cas_con_03}. This contrasts with non-toric Mori theory since we do not in addition assume that the pairing with $K_X$ is negative. Through a combinatorial interpretation of the intersection pairing between divisors and curves for toric varieties we can describe the space of curves up to numerical equivalence as relations between ray generators. Namely, for a smooth toric variety $X_\Sigma$ described by a fan $\Sigma$ in $\R^n$ we have
$$N_1(X_\Sigma)=\on{ker}(F)\quad\text{where}\quad F\colon\R^{G(\Sigma)}\to\R^n,(a_v)_{v\in G(\Sigma)}\mapsto\sum_{v\in G(\Sigma)}a_vv$$
In other words, $N_1(x_\Sigma)$ is identified with the set of relations between elements of $G(\Sigma)$.

We will often use the basic fact that a relation
\begin{equation} \tag{$\diamondsuit$} \label{eqn:relation_eg}
    a_1x_1+\dots+a_rx_r-b_1y_1-\dots-b_sy_s=0
\end{equation}
for $X_\Sigma$ with $a_i,b_j>0$ defines an effective class if $\la y_1,\dots,y_s\ra$ is a cone in $\Sigma$ \cite[Lem.~1.4]{cas_con_03}. 

We will follow the standard convention of writing relations either as in (\ref{eqn:relation_eg}) or as
$$a_1x_1+\dots+a_rx_r=b_1y_1+\dots+b_sy_s$$
with all coefficients implicitly assumed to be positive, thus interpreting this relation as `left hand side minus right hand side equals zero'.

For each primitive collection $\mP=\{v_1,\dots,v_r\}$ we obtain a relation (and hence a curve class) of the form
$$v_1+\dots+v_r-a_1y_1-\dots-a_sy_s=0$$
where $y_1,\dots,y_s$ are the generators of the minimal cone in $\Sigma$ containing $v_1+\dots+v_r$, and $a_1,\dots,a_s>0$. Relations arising in this way are called \emph{primitive relations}, they define effective curve classes, and together they generate $\on{NE}(X_\Sigma)$ as a real cone \cite{bat_cla_91}.

We call the unique primitive integral class along an extremal ray an \emph{extremal class}. We will refer often to the following influential result of Reid describing the combinatorics of fans (or polytopes) `near to' extremal classes.

\begin{thm}[\!{\cite[Thm.~2.4]{rei_dec_83}}] \label{thm:reid}
Let $\gamma$ be an extremal class for a smooth toric variety $X_\Sigma$. Then there is a primitive collection $\mP$ such that the corresponding primitive relation is identified with $\gamma$. Write this relation as
$$v_1+\dots+v_r-a_1y_1-\dots-a_sy_s=0$$
For each $z_1,\dots,z_t\in G(\Sigma)$ such that $z_i\not=x_p$ and $z_i\not=y_q$ for all $i,p,q$ and such that $\langle y_1,\dots,y_s,z_1,\dots,z_t\rangle$ is a cone in $\Sigma$, we have that
$$\langle v_1,\dots,\widehat{v_i},\dots,v_r,y_1,\dots,y_s,z_1,\dots,z_t\rangle$$
is a cone in $\Sigma$ for each $i=1,\dots,r$.
\end{thm}

Thus extremal classes produce many cones nearby to $v_1,\dots,v_r$. We will also draw on the following result of Casagrande.

\begin{cor}[\!{\cite[Cor.~4.4]{cas_con_03}}]
Let $X_\Sigma$ be a smooth toric Fano variety and let $\mP$ be a primitive collection for $\Sigma$. If the curve class corresponding to $\mP$ has anticanonical degreee $1$ then it is extremal.
\end{cor}

\section{Proof of weak conjecture}

We begin the proof of Thm.~\ref{thm:weak} with the following result.

\begin{prop} \label{prop:bound_pc}
For any integer $m \geq 1$ there is a constant $C_m$ depending only on $m$ with the property: if $X=X_\Sigma$ is a smooth toric Fano variety and $v_1, \dots, v_m\in G(\Sigma)$ form a cone in $\Sigma$, then the number of primitive collections of the form $v_1, \dots, v_m, w$ is bounded by $C_m$. 
\end{prop}

The key part of the statement is that the constant $C_m$ is independent of dimension.

\begin{proof}
We prove the statement by induction on $m$. Casagrande proved $C_1 = 3$ \cite[Lem.~2.3]{cas_tor_03}. Suppose the statement holds for any integer smaller than $m$. We will show 
$$C_m \leq \sum_{i=1}^{m-1} \left(2\cdot{{m-1}\choose{i}}+ {{m}\choose{i}}\right) C_i$$
We first show that the number of primitive relations of the form $v_1+\dots+v_m+w = b_1u_1+\dots+b_tu_t$ such that $t <m$ is bounded by
$$N_{m-1}=\sum_{i=1}^{m-1} {{m-1}\choose{i}} C_i$$
Let
$v_1+\dots+v_m+w = b_1u_1+\dots+b_tu_t$ be such a relation with $\sum_{j=1}^t b_j$ maximal. Then, by our induction hypothesis, for any subset of size $1\leq l \leq m-1$ of $\{u_1, \dots, u_t\}$ there are at most $C_l$ primitive collections of length $l+1$ containing that subset. If there are more than $N_{m-1}$ such relations, there must be a relation $v_1+\dots+v_m+w' =b_1'u_1'+\dots +b'_s u'_s$ 
such that $w'$ does not make a primitive collection with any subset of $\{u_1, \dots, u_t\}$; that is, $w',u_1, \dots, u_t$ is a cone. On the other hand we have the relation
$$\sum_{j=1}^s b'_ju'_j+w=\sum_{j=1}^t b_ju_j+w'$$
Note that $w', u_1, \dots, u_t$ form a cone and so this relation corresponds to an effective curve class of anticanonical degree $\sum_{j=1}^s b'_j-\sum_{j=1}^t b_j \leq 0$, a contradiction. 

We next show the number of primitive relations of the form $v_1+\dots+v_m+w = u_1+\dots+u_{m}$ is bounded by $N_{m-1}+\sum_{i=1}^{m-1}{{m}\choose{i}} C_i$. We suppose otherwise and fix one such relation 
$$v_1+\dots+v_m+w = u_1+\dots+u_{m}.$$
By the induction hypothesis there are at most $\sum_{i=1}^{m-1}{{m}\choose{i}} C_i$ vertices $w'$ such that $w'$ and a proper subset of $u_1, \dots, u_m$ form a primitive collection. Thus there are more than $N_{m-1}$ relations of the form  
$v_1+\dots+v_m+w' = u'_1+\dots+u'_{m}$ such that $w'$ forms a cone with every proper subset of $\{u_1, \dots, u_m\}$. Since 
$$ w + \sum_{j=1}^m u'_j = w' + \sum_{j=1}^m u_j,$$
$u_1, \dots, u_m, w'$ cannot form a cone, so they form a primitive collection and we get more than $N_{m-1}$ primitive relations of the form  
$$u_1+\dots+u_m+w' = \sum_{j=1}^r c_j x_j$$
A similar argument to the above shows that there is one such relation with $r=m$ and $c_j=1$ for all $j$. Therefore this relation is extremal and so $u_1, \dots, u_m, x_1, \dots, x_m$ form a cone by Thm.~\ref{thm:reid}. We have 
$$ \sum_{j=1}^m u_j + \sum_{j=1}^m x_j  = \sum_{j=1}^m v_j + w + \sum_{j=1}^m u_j + w' = \sum_{j=1}^m u'_j + \sum_{j=1}^m u_j +w$$
Since $w, u'_1, \dots, u'_m$ cannot form a cone there is a subset of $\{u'_{i_1}, \dots, u'_{i_l}\}$ of the $u'_j$ which form a primitive collection with $w$ and so there is a primitive relation 
$$u'_{i_1}+\dots+u'_{i_l}+w = \sum_{j=1}^p d_j y_j$$ 
with $\sum_{j=1}^p d_j \leq m$. We can then write $\sum_{j=1}^m u_j + \sum_{j=1}^m x_j$ as a linear combination of ray generators with sum of coefficients $\leq 2m$. This contradicts the assumption that $X$ is Fano.
\end{proof}

\begin{thm} \label{thm:weak_main}
For any $k\in\Z_{\geq0}$ there is a constant $A_k$ such that for any smooth toric Fano variety of dimension $n$ with a primitive relation of the form
$$v_1+\dots+v_{n+1-k}=0$$
we have $\rho(X) \leq A_k$. 
\end{thm}

Note that this follows from classification results \cite{cms_cha_02,obr_cla_07} when $k\leq1$.

\begin{lemma} \label{lem:improved}
Let $x_1,\dots,x_r\in G(\Sigma)$ be such that $\la x_1,\dots,x_r\ra\in\Sigma$. Then there exist $t\geq 1$ and $y_1,\dots,y_{t+1}\in G(\Sigma)$ such that $
\la x_1,\dots,x_r,y_1,\dots,y_t\ra\in\Sigma$ and
$$y_1+\dots+y_{t+1}\equiv0\on{mod}{\on{span}(x_1,\dots,x_r)}$$
\end{lemma}

\begin{proof}
Consider the fan for $D(x_1)$, naturally living in $\R^n/\on{span}(x_1)$, with ray generators the projections of ray generators in $G(\Sigma)$ neighbouring $x_1$ \cite[\S2]{nil_gor_05}. Repeating this construction gives the fan for $Z=\bigcap_{i=1}^rD(x_i)$ living in $\R^n/\on{span}(x_1,\dots,x_r)$ with ray generators the projections of ray generators in $G(\Sigma)$ sharing a cone with all $x_1,\dots,x_r$ simultaneously. This is a smooth toric variety and so there is a primitive collection $\ol{y}_1,\dots,\ol{y}_{t+1}$ with
$$\ol{y}_1+\dots+\ol{y}_{t+1}=\ol{0}$$
corresponding to a minimal component of degree $t+1$. Lift $\ol{y}_i$ to ray generators $y_i\in G(\Sigma)$ giving a cone $\la y_1,\dots,y_t\ra$ and the congruence
$$y_1+\dots+y_{t+1}\equiv0\on{mod}{\on{span}(x_1,\dots,x_r)}$$
It remains to show that $\sigma=\la x_1,\dots,x_r,y_1,\dots,y_t\ra\in\Sigma$. Assume not, then there is a ray $\rho=\R_{\geq0}\cdot v$ contained in the relative interior of $\sigma$; so $v=\sum_{i=1}^r\alpha_ix_i +\sum_{j=1}^t\beta_j y_j$ for $\alpha_i, \beta_j>0$. But the projection of $\sigma$ then contains the ray $\rho$ and so $\ol{y}_1,\dots,\ol{y}_t$ do not span a cone in $\Sigma_{Z}$, which is a contradiction.
\end{proof}

\begin{prop} \label{prop:mk_bound}
There is a constant $M_k$ depending only on $k$ with the following property: for any smooth toric Fano variety $X=X_\Sigma$ and any $y_1, \dots, y_k \in G(\Sigma)$, there are 
$t \leq M_k$ cones $\sigma_1, \dots, \sigma_{t} \in \Sigma$ of dimension at most $k2^k$ such that any nonnegative combination of the $y_1,\dots,y_k$ belong to at least one $\sigma_j$.
\end{prop}

\begin{proof} Suppose $y_1, \dots, y_k \in G(\Sigma)$. If $\la y_1, \dots, y_k\ra \not\in \Sigma$ then $y_1+\dots+y_k$ is in a unique minimal cone $\la z_1, \dots, z_m\ra \in \Sigma$ and so we can write 
\begin{equation}\tag{$\spadesuit$} \label{eqn:eff_class_1}
y_1+\dots +y_k = \sum_{t=1}^m a_t z_t
\end{equation}
for positive integers $a_i$, and with $\sum a_i < k$ since the class defined by this relation is effective. Let $r=k-\sum_{t=1}^m a_t$ be the anticanonical degree of this effective class.

We first prove by induction on $r$ that there is a constant $M_{k,r}$ such that for any $k$ vertices $y_1, \dots, y_k\in G(\Sigma)$ defining an effective class of anticanonical degree $r$ as above, there are at most $M_{k,r}$ cones of dimension at most $k2^r$ in $\Sigma$ covering the cone generated by $y_1,\dots,y_k$ (i.e.~their positive linear span). We always assume $y_1 \dots, y_k$ do not generate a cone in $\Sigma$ since otherwise the statement is clear. 

If $r=1$ and $\{y_1, \dots, y_k\} \cap \{z_1, \dots, z_m\} = \emptyset$, then (\ref{eqn:eff_class_1}) defines an extremal class and so $y_1, \dots, y_k$ form a primitive collection. We can thus apply Thm.~\ref{thm:reid}: for each $i$ we find that  $\langle y_1, \dots, \hat{y_i}, \dots, y_k, z_1, \dots, z_m \rangle \in \Sigma$ and therefore the cone generated by $y_1, \dots, y_k$ can be covered by these $k$ cones of dimension at most $2k-1$ in $\Sigma$. If $r=1$ and $\{y_1, \dots, y_k\} \cap \{z_1, \dots, z_m\} \neq \emptyset$, then up to renumbering the $y_i$, (\ref{eqn:eff_class_1}) gives an extremal class of anticanonical degree 1 of the form  
$$y_1+\dots + y_{k'} = \sum_{t=1}^{m'} a'_tz'_t$$ 
with $k' < k$ and $\la z'_1, \dots, z'_m,y_{k'+1}\dots, y_k \ra \in \Sigma$. Then the cone generated by $y_1, \dots, y_k$ can be covered by the cones $\langle y_1, \dots, \hat{y_i}, \dots, y_k, z'_1, \dots, z'_{m'} \rangle$, $1\leq i \leq k'$, and 
$\la y_{k'+1}, \dots, y_k\ra$, and by  Thm.~\ref{thm:reid} these cones are in $\Sigma$. Hence $M_{k,1}\leq k$.

Suppose the statement holds for any $r' < r$. For each $1 \leq i \leq k$ we consider the set of vertices $\{y_1, \dots, \wh{y}_i, \dots, y_k, z_1, \dots, z_m\}$. If this set does not form a cone in $\Sigma$ then we can write as in (\ref{eqn:eff_class_1})
\begin{equation} \tag{$\heartsuit$} \label{eqn:eff_class_2}
 y_1+ \dots +\hat{y_i} + \dots +y_k +z_1+\dots +z_m = \sum_{j=1}^s b_j w_j
 \end{equation}
for some $w_j \in G(\Sigma)$ and positive integers $b_j$. From (\ref{eqn:eff_class_1}) and (\ref{eqn:eff_class_2}) we get 
$$\sum_{t=1}^m (a_t+1)z_t = y_i + \sum_{j=1}^s b_jw_j$$
Since $\langle z_1, \dots, z_s\rangle \in \Sigma$ we identify an effective class coming from the `opposite' relation
$$y_i + \sum_{j=1}^s b_jw_j=\sum_{t=1}^m (a_t+1)z_t$$
This class has anticanonical degree $1+\sum_{j=1}^sb_j-\sum_{t=1}^m(a_t+1)>0$ and so $\sum_{j=1}^s b_j \geq \sum_{t=1}^ma_t+m$. Applying this inequality to the calculation of the anticanonical degree of (\ref{eqn:eff_class_2}), we see that
$$k+m-1-\sum_{j=1}^s b_j \leq  k+m-1-\sum_{t=1}^m (a_t+1) = k-1-\sum_{t=1}^m a_t=r-1.$$
It follows that the anticanonical degree $r'$ of the cycle 
represented by equation (\ref{eqn:eff_class_2}) is smaller than the anticanonical degree of the cycle represented by relation (\ref{eqn:eff_class_1}). Therefore, by our induction hypothesis, the cone generated by  $\{y_1, \dots, \wh{y}_i, \dots, y_k, z_1, \dots, z_m\}$ can be covered by at most $M_{k+m-1,r'}$ cones of dimension at most $(k+m-1) 2^{r'} \leq k2^k$ in $\Sigma$. Since the cone generated by 
$\{y_1, \dots, y_k\}$ can be covered by the $k$ cones generated by  $\{y_1, \dots, \wh{y}_i, \dots, y_k, z_1, \dots, z_m\}$, for $1 \leq i \leq k$, we get
$$M_{k,r} \leq k\cdot \max { M_{k+m-1, r'}: 1\leq m \leq k-1,\; 1 \leq r'\leq r-1}.$$
noting that $m<k$ from (\ref{eqn:eff_class_1}). Since $0<r<k$ we get the desired bound
$$M_k = \max { M_{k,r}: 1 \leq r \leq k-1}$$
depending only on $k$.
\end{proof}

We recursively apply Lem.~\ref{lem:improved} starting from our primitive relation $v_1+\dots+v_{n-k+1}=0$. We thus get a sequence of collections of vertices 
$\{y^1_1, \dots, y^1_{t_1+1} \}, \{y^2_1, \dots, y^2_{t_2+1} \}, \dots, \{y^l_1, \dots, y^l_{t_l+1} \}$ such that 
\begin{itemize}
\item $t_1+\dots+t_l=k$
\item $y^i_1+\dots +y_{t_i+1}^i \equiv 0 \on{mod}\on{span}(v_1, \dots, v_{n-k+1}, y^1_1, \dots, y^{1}_{t_1+1}, y^2_1, \dots, y^{i-1}_{t_{i-1}+1})$
\item $\{v_1, \dots, v_{n-k}, y^1_1, \dots, y^{1}_{t_1}, \dots, y^{l}_1, \dots, y^{l}_{t_{l}}\} \in \Sigma$ and is of dimension $n$.
\end{itemize} 
These properties imply that for every 
$w \in G(\Sigma)$ there is a nonnegative combination $\alpha_i$ of $t_i$ of the vertices $y^{i}_1, \dots, y^i_{t_{i}+1}$ such that 
$$w+\alpha_1+\dots +\alpha_l \equiv0 \on{mod}\on{span}(v_1, \dots ,v_{n-k+1})$$

Since $\alpha_1+\dots+\alpha_l$ is a nonnegative combination of $t_1+\dots+t_l=k$ elements of $G(\Sigma)$ and since there are $t_i+1$ ways we can choose $t_i$ vertices from $y^i_1, \dots, y^i_{t_i+1}$, by Prop.~\ref{prop:mk_bound}, there are $s\leq M_k (t_1+1) \dots (t_l+1)$ cones $\sigma_1, \dots, \sigma_s \in \Sigma$ of dimension at most $k 2^k$ such that each $\alpha_1+\dots+\alpha_l$ as above belongs to one of the $\sigma_j, 1 \leq j \leq s$. So for each $w\in G(\Sigma)$, 
there is $1 \leq j \leq s$ and $u_j$ in the cone  $\sigma_j$ such that
$$w+u_j \equiv 0  \on{mod} \on{span}(v_1, \dots ,v_{n-k+1})$$
So if $w\notin\{v_1,\dots,v_{n-k+1}\}$, then $w$ does not form a cone with vertices of $\sigma_j$. By Prop.~\ref{prop:bound_pc} the number of such $w \not\in \{ v_1, \dots, v_{n-k+1} \}$ is bounded by a constant which depends only on $k$, therefore the size of $G(\Sigma)  \setminus \{v_1, \dots, v_{n-k+1}\}$ is bounded by a constant which depends only on $k$.

\begin{proof}[Proof of Thm.~\ref{thm:weak_main}]
Let $X=X_\Sigma$ be a smooth toric Fano variety of dimension $n$ with a minimal component of codegree $k$. By the argument above the number of vertices of the Fano polytope for $X$ is bounded by $n+A_k$ where $A_k$ is a constant depending only on $k$, giving the result.
\end{proof}

It is very unlikely that the bounds we produce in this proof are sharp in any meaningful sense.

\section{Sharp bounds for codegree two}

We can however verify more in the codegree two case.

\begin{thm}\label{thm:codegree2}
If $X$ is a smooth toric Fano variety with a minimal component of codegree $2$ then
$$\rho(X)\leq 5.$$
\end{thm}

This bound is sharp in all dimensions, achieved by $\pr^{n-2}\times\text{dP}_6$. The statement can be verified by a direct check in dimensions $4$ and $5$; we used the Graded Ring Database \cite{grdb} and computations of Reynolds \cite{rey_fan_22} that list all primitive relations for toric Fano $4$- and $5$-folds. and so we will focus on the case $\on{dim}(X)\geq6$. The dimension $4$ case also follows from \cite[\S5]{cfh_min_14}.

To fix notation let $x_1+\dots+x_n=0$ be a primitive relation, let $S = \{x_1, \dots, x_{n-1}\}$, let $\Gamma=\on{span}_\R(S)$, and consider the projection  $\pi\colon\R^n \to\R^n/\Gamma\cong\R^{2}$.
Suppose $x,y,z \in G(\Sigma)\setminus S$ are such that $0$ is in the convex hall of $\pi(x), \pi(y), \pi(z)$; equivalently $\pi(x) \in \langle -\pi(y),-\pi(z)\rangle$. Then $\langle x,y,z \rangle \not\in \Sigma$ since a non-trivial non-negative combination of $\{x,y,z\}$ is equal to a non-negative combination of the $x_i$. So if $\langle x, y \rangle, \langle x, z\rangle, \langle y,z \rangle \in \Sigma$, then $\{x,y,z\}$ is a primitive collection. We consider the primitive relation corresponding to 
$x,y,z$ in this case. If $x+y+z=u+v$ then by Theorem \ref{thm:reid} $\langle u,x,y\rangle , \langle u,y,z\rangle , \langle u,x,z\rangle \in \Sigma$. Since $0$ is in the convex hall of at least one of the sets $\{\pi(u),\pi(x),\pi(y)\}, \{\pi(u),\pi(y),\pi(z)\}, \{\pi(u),\pi(x),\pi(z)\}$, 
we conclude $u \in S$. Similarly $v\in S$. We thus have three possibilities. 

\begin{itemize}
\item[{\bf{(a)}}] $x+y+z=x_i+x_j$ for some (possibly equal) $i,j$ 
\item[{\bf{(b)}}] $x+y+z=v$ for some $v\in G(\Sigma)$
\item[{\bf{(c)}}] $x+y+z=0$.
\end{itemize}

We will use the division into these three cases to prove Thm.~\ref{thm:codegree2}.

\begin{lemma} \label{lem:pr2}
Suppose $X=X_\Sigma$ is a smooth toric Fano variety of dimension $n\geq 6$ and $x_1 +\dots+x_{n-1}=0$ is a primitive relation in $\Sigma$. If there is  also a primitive collection in $\Sigma$ of the form $\{x,-x\}$ then $\rho(X)\leq 5$.
\end{lemma}

\begin{proof}
If there is  $y\in G(\Sigma)$ such that $\langle x, y\rangle \not\in \Sigma$, then 
$\{x,y\}$ form a primitive collection and since $X$ is Fano, we have a primitive relation of the form $x+y=0$ or $x+y=z$ for some $z\in G(\Sigma)$, so if $y \neq -x$, then $x+y \in G(\Sigma)$. 

We consider two different cases: 

\noindent {\bf{(1)}} 
There is $y \in G(\Sigma) \setminus S$ such that $y \neq -x$, $\langle x, y \rangle \not\in \Sigma$, and $z :=x+y \not \in S$, and

\noindent {\bf{(2)}} There is no such $y$ as in case (1),

and treat each case separately. 

\smallskip

\noindent Case {\bf{(1)}}: suppose there is $y$ as above.   
Since $y = (-x)+z$ we have $\langle -x, z\rangle \not\in \Sigma$. Consider the projection $\pi$ as above. Then $\pi(y),\pi(z)$ are in the same half-plane bounded by the span of $\pi(x)$. If 
$v \in G(\Sigma)\setminus (S \cup \{x,-x,y,z\})$ is such that $\pi(v)$ in the opposite half-plane, then we show $v=-y$ or $v=-z$. If $v$ does not make a cone with one  of $x$ or $-x$ then by \cite[Lemma 2.3]{cas_tor_03} we have $v=-y$ or $v=-z$. Similarly, if $v$ does not make a cone with one of $y$ or $z$ then \cite[Lemma 2.3]{cas_tor_03} gives $v=-y$ or $v=-z$. So it is enough to show $v$ cannot make a cone with each one of $x,-x,y,z$. Otherwise since  $\pi(v)$ is in one of the cones $\langle -\pi(x), -\pi(z) \rangle$, $\langle -\pi(z), -\pi(y) \rangle$, or $\langle -\pi(y), -\pi(-x) \rangle$, we get a primitive relation of one of the forms:
\begin{enumerate}
\item[(i)] $v+x+z=A$
\item[(ii)] $v+y+z=A$
\item[(iii)] $v+y+(-x)=A$
\end{enumerate}
where $A$ has degree at most $2$. So we have $A+y=v+2z$ in case (i),  $A+x=v+2z$ in case (ii), and $A+z=v+2y$ in case (iii). In each case this gives an effective class of anticanonical degree $\on{deg}(A)-2 \leq 0$ since 
$\langle v, z \rangle, \langle v, y \rangle \in \Sigma$, and hence we get a contradiction. 

Since there is at least one $v\in G(\Sigma)$ such that $\pi(v)$ is in the opposite open half plane as $\pi(y)$, at least one of $-y$ or $-z$ should be in  $G(\Sigma)$. 
If both $-y, -z \in G(\Sigma)$ then, since $x+(-z)=-y \not\in S$, applying the same argument as above to $-y,-z$ gives that there cannot be $w\in G(\Sigma) \setminus S$ with $\pi(w)$ in the same half-plane as $\pi(y)$ other than $x,-x,y,z$. We conclude $\rho(X) \leq 5$.

Now suppose $-y \in G(\Sigma), -z\not\in G(\Sigma)$. To get the desired bound it is enough to show that there is at most one $u \in G(\Sigma)\setminus (S \cup \{x,-x,y,z, -y\})$. For such $u$ we note that $\pi(u)$ must be in the same
open half-plane as $\pi(y)$. Assume to the contrary there are two such $u_1,u_2\in G(\Sigma)$. Note that in this case $\langle -y, x \rangle, \langle -x,-y\rangle \in \Sigma$. Then $\pi(u_1)$ and $\pi(u_2)$ are in one of the cones $\langle \pi(-x), \pi(y) \rangle, \langle \pi(y), \pi(x) \rangle$
and so we can assume we have two primitive relations of the form either:
\begin{itemize}
    \item[($\alpha$)] $x+(-y)+u_1=A$ and $x+(-y)+u_2=B$
    \item[($\beta$)] $x+(-y)+u_1=A$ and $(-x)+(-y)+u_2=B$
\end{itemize}
Observe that $A$ cannot have degree strictly less than $2$ since $x+u_1=A+y$ and $\langle x, u_1\rangle \in \Sigma$; outlawing cases \textbf{(b)} and \textbf{(c)}. Similarly for $B$. We conclude that $A=x_i+x_j$ and $B=x_k+x_l$ for some $i,j,k,l$.

In case ($\alpha$) we have $x_i+x_j+u_2=x_k+x_l+u_1$ but $\langle x_i,x_j,x_k,x_l\rangle \in \Sigma$ since $n\geq 6$ and so by Thm.~\ref{thm:reid} $\langle x_i,x_j,  u_2 \rangle \in \Sigma$, giving an effective class of anticanonical degree zero and hence a contradiction. Similarly, in case ($\beta$) we have $x_i+x_j+x_k+x_l=u_1+2(-y)+u_2$ but $\langle x_i,x_j,x_k,x_l \rangle \in \Sigma$, which similarly produces a contradiction. This completes the proof of case {\bf (1)}. 

\medskip

\noindent Case {\bf{(2)}}: suppose there is no $y$ as in case {\bf (1)}. We may assume that for every $v\in G(\Sigma)$ there is at most one $u\in G(\Sigma)$ such that $\langle u, v \rangle \not\in \Sigma$ and $u\neq -v$. Indeed, if $u_1, u_2$ are two such vertices then $-u_1, -u_2 \in G(\Sigma)$ -- thus $u_1,u_2 \not\in S$ -- and $v+u_1=-u_2$ so $\langle -u_1, -u_2 \rangle \not\in \Sigma$. Also 
$(-u_1)+(-u_2) \not \in S$ as if $-u_1-u_2=x_k$ then $u_1+u_2=\sum_{t\neq k} x_k$, which would give an effective class of non-positive anticanonical degree. Replacing 
$\{x,-x, y\}$ with $\{-u_2,u_2, -u_1\}$ puts us back in the situation of case (1) and so we are done. If such a $u$ exists for $v$ we denote it by $u=v'$ to indicate its relationship to $v$. Note that if $-v, v' \in G(\Sigma)$, then we can assume at least one of $v'$ and $v+v'$ is in $S$ since otherwise considering $\{v,-v,v'\}$ will return us to the situation of case {\bf (1)}. 

Pick $w_1,w_2\in G(\Sigma) \setminus (S \cup \{x,-x\})$ such that $\pi(w_1),\pi(w_2)$ are strictly in opposite half planes with respect to 
$\{x,-x\}$. We note that $w_1$ and $w_2$ each make cones with each of $x$ and $-x$. For instance, if $\langle w_1, x\rangle \not\in \Sigma$ then $w_1+x \in S$ and so $\pi(w_1)=\pi(-x)$, contradicting our assumption that $\pi(w_1)$ and $\pi(w_2)$ are strictly in opposite half planes. Note also that there is at most one vertex $x_\star \neq \pm x$ with $\pi(x_\star)$ on the line spanned by $\pi(x)$ and $\pi(-x)$, which follows from the proof of \cite[Prop.~3.8]{cfh_min_14}.

We consider two subcases:

{\bf{(2.1)}} $\langle w_1, w_2\rangle \in \Sigma$: Then $\pi(w_1+w_2)\not=0$ and either $\pi(x) \in \langle \pi(w_1),\pi(w_2) \rangle$ or $\pi(-x) \in \langle \pi(w_1),\pi(w_2) \rangle$. 
We then get a primitive relation of the form $-x+w_1+w_2=x_k+x_l$ or of the form $x+w_1+w_2=x_k+x_l$, noting that again cases \textbf{(b)} and \textbf{(c)} are not possible since this would result in an effective class of nonpositive anticanonical degree. We show in this case there cannot be any 
$y \in G(\Sigma)\setminus\{-x\}$ such that $\langle y, x \rangle \not\in \Sigma$. If there is such $y$ then $x+y=x_m$ for some $m$ by our assumption. This gives relations  of the form either $w_1+w_2+y=x_k+x_l+x_m$  or $x_m+w_1+w_2=x_k+x_l+y$. Since $n\geq6$ we have $\la x_k,x_l,x_m \ra \in \Sigma$, and applying Thm.~\ref{thm:reid} to the extremal relation $x+y=x_m$ and the cone $\la x_k,x_l,x_m\ra$ we find $\la x_k,x_l,y\ra \in \Sigma$. In each case we thus locate effective classes of anticanonical degree zero and hence a contradiction. Similarly there  cannot be any $y \in G(\Sigma)\setminus \{x\}$ such that $\langle y, -x\rangle \not\in \Sigma$. Therefore such $x_\star$ as above does not exist in this case.

If both $-w_1, w_1'\in G(\Sigma)\setminus S$ then we claim that $z=w_1+w'_1$ is not in $S$. This will yield a contradiction since, as noted above, because $\la -w_1,w_1'\ra\notin \Sigma$ we must have $z\in S$. If $z=x_m$ for some $m$ then we have either $w_2+x_m+x=x_k+x_l+w_1'$ or $w_2+x_m+(-x)=x_k+x_l+w_1'$. Applying Thm.~\ref{thm:reid} to either of the extremal relations $\pm x+w_1+w_2=x_k+x_l$ and the cone $\la x_k,x_l,x_m\ra$ gives that $\la w_2,x_m,\pm x\ra\in\Sigma$ and so the relations above -- or rather their negatives -- define effective classes of anticanonical degree zero, which is impossible.

Therefore we can assume that at most one of $-w_1,w_1'$ is in $G(\Sigma)\setminus S$ and by a similar argument that at most one of $-w_2,w_2'$ is in $G(\Sigma)\setminus S$. To get the desired bound it is enough to show 
$$G(\Sigma) \subseteq S \cup \{x,-x,w_1,w_1',w_2,w_2',-w_1,-w_2\}$$
If $v\in G(\Sigma) \setminus(S \cup \{x,-x,w_1,w_1',w_2,w_2',-w_1,-w_2\})$ then we have a primitive relation of one of the forms $v+(\pm x) +w_1=x_i+x_j$ or $v+(\pm x) +w_2=x_i+x_j$ for some $i,j$. Indeed, $v$ forms a cone with each of $\pm x$, $w_1$ and $w_2$ by construction and $\pi(v)$ must live in one of the cones $\langle \pi(x), \pi(w_1) \rangle, \langle \pi(w_1), \pi(-x)\rangle, \langle \pi(-x), \pi(w_2) \rangle, \langle \pi(w_2),\pi(x) \rangle$. We then deduce that we are in case \textbf{(a)} similarly to before. Assume that the primitive relation we obtain is $v+x+w_1=x_i+x_j$. Combining this with one of the relations $\pm x+w_1+w_2=x_k+x_l$ gives either $w_2+x_i+x_j=v+x_k+x_l$ or $v+2w_1+w_2=x_i+x_j+x_k+x_l$, but the vertices on both right-hand sides form cones in $\Sigma$ (since $n\geq 6$) and so we obtain a contradiction. Similar arguments work in the other cases.

{\bf{(2.2)}} $\langle w_1,w_2\rangle \not\in \Sigma$: As we can apply the argument from case (1) to any two vertices whose images are in opposite half-planes we may assume that any two such vertices do not form a cone and hence form a primitive collection. Therefore, there are at most two other vertices whose images under $\pi$ are in the same half-plane as $\pi(w_1)$: namely $-w_2$ and $w'_2$. Similarly, if there are two other vertices whose images under $\pi$ are in the same half plane as $\pi(w_2)$ then they must be $-w_1$ and $w'_1$. 

If $w_1 = -w_2$ then either $G(\Sigma) \subseteq S \cup \{x,-x,x_\star,w_1,w_2\}$ and we are done or there is $v\in G(\Sigma)$ such that for some $i\in\{1,2\}$ we have that $\pi(v)$ is in the opposite half-plane as $\pi(w_i)$ and $v\neq -w_i$; that is, $v=w_i'$. Replacing $\{w_1,w_2\}$ with $\{w_i,v\}$ in this case we may assume from the beginning that $w_1\neq -w_2$. If $\{w_1,-w_1,w_2,-w_2\} \subseteq G(\Sigma)$ then since $\pi(-w_1), \pi(-w_2)$ belong to opposite half-planes, 
we have $\langle -w_1, -w_2\rangle \not\in \Sigma$ and so $(-w_1)+(-w_2) \in G(\Sigma) \setminus S$. This follows since $w_1+w_2 \in G(\Sigma)$ and  $-w_1-w_2 \not\in S$. Considering the set $\{w_1,-w_1,w_2\}$ places us back in the situation of case {\bf (1)} and we are done. Since we have seen that $G(\Sigma)\setminus (S\cup\{x,-x,x_*\}) \subseteq \{w_1,-w_1,w_2,-w_2\}$ we find that if at most three of $\{w_1,-w_1,w_2,-w_2\}$ are in $G(\Sigma)$ then $\rho(X) \leq 5$, which is the only remaining possibility.
\end{proof}

\begin{proof}[Proof of Thm.~\ref{thm:codegree2} for $n\geq 6$]

Let $\pi$ denote the projection map as before. We first prove the statement when there are  $x,y,z \in G(\Sigma)\setminus S$ such that $\langle x,y\rangle , \langle x,z\rangle, \langle y,z\rangle \in \Sigma$ and $0$ is in the convex hall of $\pi(x), \pi(y), \pi(z)$; equivalently if $\pi(z) \in \langle -\pi(x),-\pi(y) \rangle$. In this case $\{x,y,z\}$ form a primitive collection. 

By Lemma \ref{lem:pr2} we may assume that for every $v \in G(\Sigma)$ there is at most one $w\in G(\Sigma)$ with $\langle w, v \rangle \not\in \Sigma$ and if such $w$ exists, then $w\neq -v$. Indeed, if there are $w_1\neq w_2$ 
such that $\langle w_1, v\rangle, \langle w_2,v\rangle \not\in \Sigma$ and $w_1,w_2\neq -v$, then $-w_1,-w_2\in G(\Sigma)$, so $\rho(X) \leq 5$ by Lemma \ref{lem:pr2}. For $v \in G(\Sigma)$, if there is a $w\neq -v$ which does not make a cone with $v$ then we denote it by $w=v'$. We consider the cases \textbf{(a)}, \textbf{(b)} and \textbf{(c)} separately for the primitive collection $\{x,y,z\}$.

\medskip

\noindent {\bf{(a)}} Let $x+y+z=x_i+x_j$. If $G(\Sigma) \subseteq S\cup\{x,y,z,x',y',z'\}$ we are done. Assume there is $u \in G(\Sigma) \setminus (S\cup \{x,y,z,x',y',z'\})$. Then $\pi(u)$ belongs to one of the cones
$$\langle -\pi(x), -\pi(y) \rangle, \langle -\pi(y), -\pi(z) \rangle, \langle -\pi(z), -\pi(x) \rangle.$$
Without loss of generality we assume $\pi(u)\in \langle -\pi(x), -\pi(y) \rangle$ so $\{u,x,y\}$ form a primitive collection. Let $$u+x+y=B$$ be the corresponding primitive relation 
where $B=x_k+x_l$ or $B=v\in G(\Sigma)$ or $B=0$.

$B$ cannot be of the form $x_k+x_l$ as otherwise we find $z+x_k+x_l=u+x_i+x_j$, which gives an effective class of anticanonical degree zero. Indeed, $\langle z, x_k, x_l \rangle \in \Sigma$ by applying Thm.~\ref{thm:reid} to the relation $x+y+z=x_i+x_j$ and the cone $\la x_i,x_j,x_k,x_l\ra$ (noting that $n \geq 6$ is required here). We next consider the two cases $B=v$ and $B=0$.

\smallskip

{\bf{(a.1)}} Suppose $B=v$. If there is $w\in G(\Sigma)\setminus (S\cup \{x,y,z,x',y',u,u'\})$ then we get a primitive relation of the form 
$w+x+y=C$ or $w+u+x=C$ or $w+u+y=C$ where $C=0$ or $C=\alpha\in G(\Sigma)$. Note that $C=x_k+x_l$ is not possible for the same reason as above. We get $v+w=C+u$ or $v+w=C+y$ or $v+w=C+x$ and so $w=v'$. 
Therefore $G(\Sigma) \subseteq S\cup\{x,y,z,x',y',u,u',v'\}$. We show 
$x'$ and $y'$ cannot exist in $G(\Sigma)$ and so $\rho(X) \leq 5$. Suppose to the contrary that $x'\in G(\Sigma)$ and $x+x'=\beta\in G(\Sigma)$. We will find a contradiction by showing $\beta \not\in S \cup \{x,y,z,x',y',u,u',v'\}$. From $x+y+z=x_i+x_j$ we get $$\beta+y+z=x'+x_i+x_j.$$
This shows $\beta \not\in S$ since otherwise by Theorem \ref{thm:reid} the left hand side will form a cone which is not possible. Also $\beta\neq y$ or $z$ since if, for example, $\beta=y$ we get $2y+z=x'+x_i+x_j$ which is not possible since the left hand side forms a cone. From the relation $u+x+y=v$ we get $u+\beta+y=v+x'$. If $\beta=u'$ then $\gamma := \beta+u \in G(\Sigma)$, 
so $\gamma+y=v+x'$ which is not possible since $\langle v,x'\rangle \in \Sigma$ as $v\neq x$. Similarly $\beta \neq y'$. Also, $\beta\neq u$ since otherwise $2u+y=v+x'$ which is not possible since $\langle u, y\rangle \in \Sigma$. Finally $\beta\neq v'$ since otherwise, $u+v'+y=v+x'$ but this is a primitive relation so $\langle v',v\rangle$ should be in $\Sigma$, a contradiction.

{\bf{(a.2)}} Suppose $B=0$. Then we have $u=-x-y$. It is enough to consider the case
$$G(\Sigma) \subseteq S \cup \{x,y,z,x',y',z',-x-y,-x-z,-y-z\}.$$
We show that in this case $x',y',z'$, if they exist, are elements of $S$. Suppose to the contrary that for example $z' \in G(\Sigma) \setminus S$ and let $w=z+z'$. We get $w+x+y=x_j+x_j+z'$. Suppose $w=x_k\in S$. Then we have $x_k+x+y=x_i+x_j+z'$ but $\langle x_k, x, y \rangle \in \Sigma$ from applying Thm.~\ref{thm:reid} to $x+y+z=x_i+x_j$ and the cone $\la x_i,x_j,x_k\ra$, giving an effective class of anticanonical degree zero. Hence $w\notin S$. Also, $w\neq y$ since otherwise $x+2y=x_i+x_j+z'$ but $\langle x, y \rangle \in \Sigma$, giving a contradiction in the same way. Similarly $w \neq x$. We have $w\neq -y-z$ as otherwise $2z+z'+y=0$ which is not possible since $\{z',y\}$ is a part of a basis for the lattice. Similarly $w\neq -x-z$. Finally $w\neq -x-y$ as otherwise $x_i+x_j+z'=0$, giving $z' \in S$. Thus we must have $w=x'$ or $w=y'$. Assume $w=x'$. Applying the same argument as above to $x'$ we get $x+x'=y'$ or $x+x'=z'$. If 
$x+x'=z'$, then $x+z=0$, a contradiction. If $x+x'=y'$ then applying the same argument to $y'$ we get $y'+y=z'$ and hence $x+y+z=0$, a contradiction.

\medskip

\noindent {\bf{(b)}} Let $x+y+z=v$. By the proof for \textbf{(a)} we may assume for any $u_1,u_2,u_3\in G(\Sigma)\setminus S$ such that $\pi(u_1) \in \langle -\pi(u_2), -\pi(u_3) \rangle$ and 
$\{u_1,u_2,u_3\}$ is a primitive collection that we have $u_1+u_2+u_3=0$ or $u_1+u_2+u_3 \in G(\Sigma)$. 
 
If $G(\Sigma)\subseteq S \cup \{x,y,z,x',y',z'\}$ then $\rho(X) \leq 5$, so assume that there exists $u\in G(\Sigma)\setminus (S \cup \{x,y,z,x',y',z'\})$ and without loss of generality assume $\pi(u) \in \langle -\pi(x), -\pi(y) \rangle$. Then $\{u,x,y\}$ is a primitive collection with primitive relation $u+x+y=w\in G(\Sigma)$ or $u+x+y=0$. Therefore $u+v=w+z$ or $u+v=z$, so $u=v'$. So 
$G(\Sigma) \subset S \cup \{x,y,z,x',y',z',v'\}$. Note that if $v\not\in S\cup\{x',y',z'\}$, then letting $u=v$ in the above argument we get 
$2v=w+z$ or $2v=z$ which is not possible. So $v\in\{x',y',z'\}$ or 
$v\in S$. If $v\in\{x',y',z'\}$, then $v'\in\{x,y,z\}$, so $G(\Sigma)\subset S \cup \{x,y,z,x',y',z'\}$ and we are done. So we may assume $v=x_m\in S$. A similar argument shows that we may assume for any $u_1,u_2,u_3\in G(\Sigma)\setminus S$ such that $\pi(u_1) \in \langle -\pi(u_2), -\pi(u_3) \rangle$ and 
$\{u_1,u_2,u_3\}$ is a primitive collection we have $u_1+u_2+u_3=0$ or $u_1+u_2+u_3 \in S$. 

It is thus enough to show it is not possible that $x,y,z,x',y',z',x_m'$ are all distinct vertices in $G(\Sigma)\setminus S$. Suppose otherwise. We can assume $-\pi(x'_m) \in \langle \pi(x),\pi(y)\rangle$ and so we have a primitive relation of the form $x'_m+x+y=x_k$ or $x'_m+x+y=0$. The first option gives $x'_m+x_m=z+x_k$, yielding $z'=x_k$ and thus a contradiction as $z'\notin S$. As a result we must have $x'_m+x+y=0$ and $x_m + x_m' = z$. We show 
$\pi(z') \not\in  \langle -\pi(x), -\pi(y) \rangle$. If not, we obtain a primitive relation $z'+x+y=0$ or $z'+x+y=x_j$, giving $z'+x_m=z$ or $z'+x_m=z+x_j$. Note that $\langle z',x_m\rangle \not\in \Sigma$ and so $z=x_m$, a contradiction. Hence $\pi(z') \not\in  \langle -\pi(x), -\pi(y) \rangle$ and therefore $\pi(z') \in  \langle -\pi(x), -\pi(z) \rangle = \langle -\pi(x), -\pi(x'_m) \rangle$ or $\pi(z') \in  \langle -\pi(y), -\pi(z) \rangle = \langle -\pi(y), -\pi(x'_m) \rangle$. Without loss of generality we may assume $\pi(z') 
 \in  \langle -\pi(y), -\pi(x_m') \rangle$. Then we 
have a  primitive relation $z'+x_m'+y=0$ or $z'+x_m'+y=x_l$. The first gives $z'=x$, a contradiction. The second gives $z'=x+x_l$ which is also a contradiction since $x'\not\in S$. Hence $\rho(X) \leq 5$.

\medskip

\noindent {\bf{(c)}} Let $x+y+z=0$. Suppose that there exists some $v \in G(\Sigma)\setminus(S\cup\{x,y,z,x',y',z'\})$. Then we can assume without loss of generality that 
$\pi(v) \in \langle -\pi(y), -\pi(z) \rangle$. So $\{v,y,z\}$ is a primitive collection. If $v+y+z \neq 0$ then by cases \textbf{(a)} and \textbf{(b)} we have $\rho(X)\leq 5$. If $v+y+z=0$ then $v=x$, a contradiction. Hence $\rho(X) \leq 5$. 

\medskip

To summarise, we have reduced to the case where if $\la x,y\ra\in \Sigma$ and $\pi(z) \in \langle -\pi(x),-\pi(y) \rangle$ then $z$ does not make a cone with one of $x$ or $y$. 
By Lemma \ref{lem:pr2} we can assume for any $x \in G(\Sigma), -x \not\in G(\Sigma)$.  We select $x,y \in G(\Sigma)$ with $\la x,y\ra \in \Sigma$ and such that the cone generated by $\pi(x)$ and $\pi(y)$ is maximal among cones in $\R^2$ coming from such pairs. If there is $z\in G(\Sigma)$ such that $\pi(z)$ is outside the cone  $\langle \pi(x), \pi(y) \rangle$, then $z$ does not form a cone with either $x$ or $y$: if  $\pi(z)$ is in $\langle \pi(x), -\pi(y) \rangle$ or $\la -\pi(x),\pi(y)\ra$ then this follows from maximality of the cone generated by $\pi(x)$ and $\pi(y)$. Thus $z=y'$ or $z=x'$. 

Let $v$ be such that $\pi(v)$ is in the open half plane determined by $\pi(x)$ not containing $\pi(y)$. If $v=x'$, then $x+x'=y'$ since $\pi(x+x')$ is non-zero and is outside $\langle \pi(y),\pi(x) \rangle.$ So the cone $\la\pi(x),\pi(y)\ra$ is covered by the the cones $\la -\pi(x), -\pi(y') \ra$, $\la -\pi(y'), -\pi(x') \ra$, and $\la -\pi(x'), -\pi(y) \ra$. So we see that $G(\Sigma)\subseteq S\cup\{x,y,x',y'\}$, which concludes the proof. A similar argument shows the result if $v=y'$ and $\pi(v)$ and 
$\pi(x)$ are in opposite half planes determined by $\pi(y)$. If $v=y'$ and $\pi(v)$ and $\pi(x)$ are in the same half plane determined by $\pi(y)$, then we choose $w$ such that $\pi(w)$ is in the opposite open half plane and use the same argument to conclude the proof. 
\end{proof}

\begin{cor} \label{cor:codegree2}
Conjecture \ref{conj:strong} and the conjecture of Chen--Fu--Hwang hold for all smooth toric Fano varieties with a minimal component of codegree $2$.
\end{cor}

Conjecture \ref{conj:strong} predicts that $\rho(X)\leq 6$ for all such $X$, and the conjecture of Chen--Fu--Hwang predicts $\rho(X)\leq\lfloor\frac{n(n+1)}{2(n-2)}\rfloor$, which is $5$ for $n=4,5,6,7$; $6$ for $n=3,8,9,10$; and in general increases linearly with $n$.

\bibliographystyle{acm}
\bibliography{bw}

\end{document}